\documentclass[12pt]{amsart}

\usepackage{graphicx,amsmath,amssymb,amsthm,enumerate}
\usepackage[colorlinks=true,linkcolor=black,citecolor=black,backref=page]{hyperref}
\usepackage[all]{xy}

\newtheorem{thm}{\bf Theorem}[section]
\newtheorem{lem}[thm]{\bf Lemma}
\newtheorem{cor}[thm]{\bf Corollary}
\newtheorem{prop}[thm]{\bf Proposition}

\theoremstyle{definition}

\theoremstyle{remark}
\newtheorem{quest}[thm]{\bf Question}
\newtheorem{rem}[thm]{Remark}
\newtheorem{ex}[thm]{Example}

\DeclareMathOperator{\chara}{char}
\DeclareMathOperator{\Tor}{Tor}
\DeclareMathOperator{\reg}{reg}
\DeclareMathOperator{\depth}{depth}
\DeclareMathOperator{\astab}{astab}

\DeclareMathOperator{\Ass}{Ass}
\DeclareMathOperator{\Hom}{Hom}
\DeclareMathOperator{\Min}{Min}
\DeclareMathOperator{\htt}{ht}

\DeclareMathOperator{\gr}{gr}

\DeclareMathOperator{\Rees}{Rees}
\DeclareMathOperator{\grade}{grade}

 \newcommand{\mm}{{\mathfrak m}}
 \newcommand{\pp}{{\mathfrak p}}
 \newcommand{\qq}{{\mathfrak q}}

\title[Associated primes of powers]{Powers of sums and their associated primes}
\author{Hop D. Nguyen}
\address{Institute of Mathematics, VAST, 18 Hoang Quoc Viet, Cau Giay, 10307 Hanoi, Vietnam}
\email{ngdhop@gmail.com}
\author{Quang Hoa Tran}
\address{University of Education, Hue University, 34 Le Loi St., Hue City, Viet Nam}
\email{tranquanghoa@hueuni.edu.vn}

\begin{document}

\begin{abstract}
Let $A, B$ be polynomial rings over a field $k$, and $I\subseteq A, J\subseteq B$ proper homogeneous ideals. We analyze the associated primes of powers of $I+J\subseteq A\otimes_k B$ given the data on the summands. The associated primes of large enough powers of $I+J$ are determined. We then answer positively a question due to I. Swanson and R. Walker about the persistence property of $I+J$ in many new cases. 
\end{abstract}

\maketitle

\section{Introduction}
\label{sect_intro}
Inspired by work of Ratliff \cite{R}, Brodmann \cite{Br} proved that in any Noetherian ring, the set of associated primes of powers of an ideal is eventually constant for large enough powers. Subsequent work by many researchers have shown that important invariants of powers of ideals, for example, the depth and the Castelnuovo--Mumford regularity also eventually stabilize in the same manner. For a recent survey on associated primes of powers and related questions, we refer to Hoa's paper \cite{Hoa20}.

Our work is inspired by the aforementioned result of Brodmann, and recent studies about powers of sums of ideals \cite{HNTT, HTT, NgV}. Let $A, B$ be standard graded polynomial rings over a field $k$, and $I\subseteq A, J\subseteq B$ proper homogeneous ideals. Denote $R=A\otimes_k B$ and $I+J$ the ideal $IR+JR$. Taking sums of ideals this way corresponds to the geometric operation of taking fiber products of projective schemes over the field $k$. In \cite{HNTT,HTT,  NgV}, certain homological invariants of powers of $I+J$, notably the depth and regularity, are computed in terms of the corresponding invariants of powers of $I$ and $J$. In particular, we have exact formulas for $\depth R/(I+J)^n$ and $\reg R/(I+J)^n$ if either $\chara k=0$, or $I$ and $J$ are both monomial ideals. It is therefore natural to ask:
\begin{quest}
\label{quest_Assformula}
Is there an exact formula for $\Ass(R/(I+J)^n)$ in terms of the associated primes of powers of $I$ and $J$?  
\end{quest}
The case $n=1$ is simple and well-known: Using the fact that $R/(I+J)\cong (A/I)\otimes_k (B/J)$, we deduce (\cite[Theorem 2.5]{HNTT}):
\[
\Ass R/(I+J)=\mathop{\bigcup_{\pp \in \Ass_A(A/I)}}_{\qq \in \Ass_B(B/J)} \Min_R(R/\pp+\qq).
\]
Unexpectedly, in contrast to the case of homological invariants like depth or regularity, we do not have a complete answer to Question \ref{quest_Assformula} in characteristic zero yet. One of our main results is the following partial answer to this question.
\begin{thm}[Theorem \ref{thm_Asscontainments}]
\label{thm_main1}
Let $I$ be a proper homogeneous ideal of $A$ such that $\Ass(A/I^n)=\Ass(I^{n-1}/I^n)$ for all $n\ge 1$. Let $J$ be any proper homogeneous ideal of $B$. Then for all $n\ge 1$, there is an equality
\[
\Ass_R \frac{R}{(I+J)^n}  =   \bigcup_{i=1}^n \mathop{\bigcup_{\pp \in \Ass_A(A/I^i)}}_{\qq \in \Ass_B(J^{n-i}/J^{n-i+1})} \Min_R(R/\pp+\qq).
\]
If furthermore $\Ass(B/J^n)=\Ass(J^{n-1}/J^n)$ for all $n\ge 1$, then for all such $n$, there is an equality
\[
\Ass_R \frac{R}{(I+J)^n}  =   \bigcup_{i=1}^n \mathop{\bigcup_{\pp \in \Ass_A(A/I^i)}}_{\qq \in \Ass_B(B/J^{n-i+1})} \Min_R(R/\pp+\qq).
\]
\end{thm}
The proof proceeds by filtering $R/(I+J)^n$ using exact sequences with terms of the form $M\otimes_k N$, where $M,N$ are nonzero finitely generated modules over $A,B$, respectively, and applying the formula for $\Ass_R(M\otimes_k N)$.

Concerning Theorem \ref{thm_main1},  the hypothesis $\Ass(A/I^n)=\Ass(I^{n-1}/I^n)$ for all $n\ge 1$ holds in many cases, for example, if $I$ is a monomial ideal of $A$, or if $\chara k=0$ and $\dim(A/I)\le 1$ (see Theorem \ref{thm_specialcase_ass} for more details). We are not aware of any ideal in a polynomial ring which does not satisfy this condition (over non-regular rings, it is not hard to find such an ideal). In characteristic zero, we show that the equality $\Ass(A/I^n)=\Ass(I^{n-1}/I^n)$ holds for all $I$ and all $n$ if $\dim A\le 3$. If $\chara k=0$ and $A$ has Krull dimension four, using the Buchsbaum--Eisenbud structure theory of perfect Gorenstein ideals of height three and  work by Kustin and Ulrich \cite{KU}, we establish the equality $\Ass(A/I^2)=\Ass(I/I^2)$ for all $I\subseteq A$ (Theorem \ref{thm_dim4_nequals2}).

Another motivation for this work is the so-called persistence property for associated primes. The ideal $I$ is said to have the \emph{persistence property} if $\Ass(A/I^n)\subseteq \Ass(A/I^{n+1})$ for all $n\ge 1$. Ideals with this property abound, including for example complete intersections. The persistence property has been considered by many people; see, e.g., \cite{FHV, HQu, KSS, SW}. As an application of Theorem \ref{thm_main1}, we prove that if $I$ is a monomial ideal satisfying the persistence property, and $J$ is any ideal, then $I+J$ also has the persistence property (Corollary \ref{cor_persistence}). Moreover, we generalize previous work due to I. Swanson and R. Walker \cite{SW} on this question: If $I$ is an ideal such that $I^{n+1}:I=I^n$ for all $n\ge 1$, then for any ideal $J$ of $B$, $I+J$ has the persistence property (see Corollary \ref{cor_persistence}(ii)). In \cite[Corollary 1.7]{SW}, Swanson and Walker prove the same result under the stronger condition that $I$ is normal. It remains an open question whether for any ideal $I$ of $A$ with the persistence property and any ideal $J$ of $B$, the sum $I+J$ has same property.

The paper is structured as follows. In Section \ref{sect_assofquotient}, we provide large classes of ideals $I$ such that the equality $\Ass(A/I^n)=\Ass(I^{n-1}/I^n)$ holds true for all $n\ge 1$. An unexpected outcome of this study is a counterexample to \cite[Question 3.6]{AM}, on the vanishing of the map $\Tor^A_i(k,I^n)\to \Tor^A_i(k,I^{n-1})$. Namely in characteristic 2, we construct a quadratic ideal $I$ in $A$ such that the natural map $\Tor^A_*(k,I^2)\to \Tor^A_*(k,I)$ is not zero (even though $A/I$ is a Gorenstein Artinian ring, see Example \ref{ex_non-Torvanishing}). This example might be of independent interest, for example, it gives a negative answer to a question in \cite{AM}. Using the results in Section \ref{sect_assofquotient}, we give a set-theoretic upper bound and a lower bound for $\Ass(R/(I+J)^n)$ (Theorem \ref{thm_Asscontainments}). Theorem \ref{thm_Asscontainments} also gives an exact formula for the asymptotic primes of $I+J$ without any condition on $I$ and $J$. In the last section, we apply our results to the question on the persistence property raised by Swanson and Walker.
\section{Preliminaries} 
\label{sect_prelim}
For standard notions and results in commutative algebra, we refer to the books \cite{BH, Eis}.

Throughout the section, let $A$ and $B$ be two commutative Noetherian algebras over a field $k$ such that $R = A\otimes_k B$ is also Noetherian. Let $M$ and $N$ be two nonzero finitely generated modules over $A$ and $B$, respectively. Denote by $\Ass_A M$ and $\Min_A M$ the set of associated primes and minimal primes of $M$ as an $A$-module, respectively.  

By a filtration of ideals $(I_n)_{n\ge 0}$ in $A$, we mean the ideals $I_n, n\ge 0$ satisfies the conditions $I_0=A$ and $I_{n+1} \subseteq I_n$ for all $n\ge 0$. Let $(I_n)_{n\ge 1}$ and $(J_n)_{n\ge 1}$ be filtrations of ideals in $A$ and $B$, respectively. Consider the filtration $(W_n)_{n\ge 0}$ of $A\otimes_k B$ given by  $W_n= \sum_{i=0}^n I_iJ_{n-i}$. The following result is useful for the discussion in Section \ref{sect_asspr_powers}.
\begin{prop} [{\cite[Lemma 3.1, Proposition 3.3]{HNTT}}] 
\label{prop_decomposition}
For arbitrary ideals $I\subseteq A$ and $J\subseteq B$, we have $I\cap J=IJ$. Moreover with the above notation for filtrations, for any integer $n \ge 0$, there is an isomorphism
$$\displaystyle W_n/W_{n+1}
\cong \bigoplus_{i=0}^n\big(I_i/I_{i+1} \otimes_k J_{n-i}/J_{n-i+1}\big).$$
\end{prop}
We recall the following description of the associated primes of tensor products; see also \cite[Corollary 3.7]{STY}.
\begin{thm}[{\cite[Theorem 2.5]{HNTT}}] 
\label{thm_asso}
Let $M$ and $N$ be nonzero finitely generated modules over $A$ and $B$, respectively. Then there is an equality
\[
\Ass_R(M \otimes_k N) = \displaystyle \bigcup_{\begin{subarray}{l} \pp\in \Ass_A(M)\\
 \qq \in \Ass_B(N)\end{subarray}} \Min_R(R/\pp+\qq).
\]
\end{thm}

The following simple lemma turns out to be useful in the sequel.
\begin{lem}
\label{lem_colon_containment}
Assume that $\chara k=0$. Let $A=k[x_1,\ldots,x_r]$ be a standard graded polynomial ring over $k$,  and $\mm$ its graded maximal ideal. Let $I$ be proper homogeneous ideal of $A$. Denote by $\partial(I)$ the ideal generated by partial derivatives of elements in $I$. Then there is a containment $I:\mm \subseteq \partial(I).$

In particular, $I^n:\mm \subseteq I^{n-1}$ for all $n\ge 1$. If for some $n\ge 2$, $\mm\in \Ass(A/I^n)$ then $\mm \in \Ass(I^{n-1}/I^n)$.
\end{lem}
\begin{proof}
Take $f\in I:\mm$. Then $x_if\in I$ for every $i=1,\ldots,r$. Taking partial derivatives, we get $f+x_i(\partial f/\partial x_i) \in \partial(I)$. Summing up and using Euler's formula, $(r+\deg f)f\in \partial(I)$. As $\chara k=0$, this yields $f\in \partial(I)$, as claimed.

The second assertion holds since by the product rule, $\partial(I^n)\subseteq I^{n-1}$.

If $\mm \in \Ass(A/I^n)$ then there exists an element $a\in (I^n:\mm)\setminus I^n$. Thus $a\in I^{n-1}\setminus I^n$, so $\mm \in \Ass(I^{n-1}/I^n)$.
\end{proof}
The condition on the characteristic is indispensable: The inclusion $I^2:\mm \subseteq I$ may fail in positive characteristic; see Example \ref{ex_non-Torvanishing}.

The following lemma will be employed several times in the sequel. Denote by $\gr_I(A)$ the associated graded ring of $A$ with respect to the $I$-adic filtration.
\begin{lem}
\label{lem_sufficient_forpersistence}
Let $A$ be a Noetherian ring, and $I$ an ideal. Then the following are equivalent:
\begin{enumerate}[\quad \rm (i)]
\item $I^{n+1}:I=I^n$ for all $n\ge 1$,
\item $(I^{n+1}:I)\cap I^{n-1}=I^n$ for all $n\ge 1$,
\item $\depth \gr_I(A)>0$,
\item $I^n=\widetilde{I^n}$ for all $n\ge 1$, where $\widetilde{I}=\bigcup\limits_{i\ge 1}(I^{i+1}:I^i)$ denotes the Ratliff-Rush closure of $I$.
\end{enumerate}
If one of these equivalent conditions holds, then $\Ass(A/I^n)\subseteq \Ass(A/I^{n+1})$ for all $n\ge 1$, namely $I$ has the \textup{persistence property}.
\end{lem}

\begin{proof}
Clearly (i) $\Longrightarrow$ (ii). We prove that (ii) $\Longrightarrow$ (i).

Assume that $(I^{n+1}:I)\cap I^{n-1}=I^n$ for all $n\ge 1$. We prove by induction on $n\ge 1$ that $I^n:I=I^{n-1}$.

If $n=1$, there is nothing to do. Assume that $n\ge 2$. By the induction hypothesis, $I^n:I \subseteq I^{n-1}:I=I^{n-2}$. Hence $I^n:I=(I^n:I)\cap I^{n-2}=I^{n-1}$, as desired.

That (i) $\Longleftrightarrow$ (iii) $\Longleftrightarrow$ (iv) follows from \cite[(1.2)]{HLS} and \cite[Remark 1.6]{RS}.

The last assertion follows from \cite[Section 1]{HQu}, where the property $I^{n+1}:I=I^n$ for all $n\ge 1$, called the \emph{strong persistence property} of $I$, was discussed.
\end{proof}

\section{Associated primes of quotients of consecutive powers}
\label{sect_assofquotient}

The following question is quite relevant to the task of finding the associated primes of powers of sums.
\begin{quest}
\label{quest_ass}
Let $A$ be a standard graded polynomial ring over a field $k$ (of characteristic zero), and $I$ a proper homogeneous ideal. Is it true that
\[
\Ass(A/I^n)=\Ass(I^{n-1}/I^n) \quad \text{for all $n\ge 1$?}
\]
\end{quest}

We are not aware of any ideal not satisfying the equality in Question \ref{quest_ass} (even in positive characteristic). In the first main result of this paper, we provide some evidence for a positive answer to Question \ref{quest_ass}. Denote by $\Rees(I)$ the Rees algebra of $I$. The ideal $I$ is said to be \emph{unmixed} if it has no embedded primes. It is called \emph{normal} if all of its powers are integrally closed ideals.
\begin{thm}
 \label{thm_specialcase_ass}
Question \ref{quest_ass} has a positive answer if any of the following conditions holds:
\begin{enumerate}[\quad \rm(1)]
 \item $I$ is a monomial ideal.
 \item $\depth \gr_I(A) \ge 1$.
\item $\depth \Rees(I)\ge 2$.
 \item $I$ is normal.
 \item $I^n$ is unmixed for all $n\ge 1$, e.g. $I$ is generated by a regular sequence.
 \item All the powers of $I$ are primary, e.g. $\dim(A/I)=0$.
 \item $\chara k=0$ and $\dim(A/I)\le 1$.
 \item $\chara k=0$ and $\dim A\le 3$.
\end{enumerate}
\end{thm}
\begin{proof}
(1):  See \cite[Lemma 4.4]{MV}. 

(2): By Lemma \ref{lem_sufficient_forpersistence}, since $\depth \gr_I(A)\ge 1$, $I^n:I=I^{n-1}$ for all $n\ge 1$. Induce on $n\ge 1$ that $\Ass(A/I^n)=\Ass(I^{n-1}/I^n)$. 

Let $I=(f_1,\ldots,f_m)$. For $n\ge 2$, as $I^n:I=I^{n-1}$, the map 
$$
I^{n-2} \to \underbrace{I^{n-1}\oplus \cdots \oplus I^{n-1}}_{m \, \text{times}}, a\mapsto (af_1,\ldots,af_m),
$$
induces an injection
\[
\frac{I^{n-2}}{I^{n-1}} \hookrightarrow \left(\frac{I^{n-1}}{I^n}\right)^{\oplus m}.
\]
Hence $\Ass(A/I^{n-1})=\Ass(I^{n-2}/I^{n-1}) \subseteq \Ass(I^{n-1}/I^n)$. The exact sequence
\[
0\to I^{n-1}/I^n \to A/I^n \to A/I^{n-1} \to 0
\]
then yields $\Ass(A/I^n)\subseteq \Ass(I^{n-1}/I^n)$, which in turn implies the desired equality.

Next we claim that (3) and (4) all imply (2).

(3) $\Longrightarrow$ (2): This follows from a result of Huckaba and Marley \cite[Corollary 3.12]{HuMa} which says that either $\gr_I(A)$ is Cohen-Macaulay (and hence has depth $A= \dim A$), or $\depth \gr_I(A)=\depth \Rees(I)-1$.

(4) $\Longrightarrow$ (2): If $I$ is normal, then $I^n:I=I^{n-1}$ for all $n\ge 1$. Hence we are done by Lemma \ref{lem_sufficient_forpersistence}.

(5): Take $P\in \Ass(A/I^n)$, we show that $P\in \Ass(I^{n-1}/I^n)$. Since $A/I^n$ is unmixed, $P \in \Min(A/I^n)= \Min(I^{n-1}/I^n)$.

Observe that (6) $\Longrightarrow$ (5).

(7): Because of (6), we can assume that $\dim(A/I)=1$. Take $P\in \Ass(A/I^n)$, we need to show that $P\in \Ass(I^{n-1}/I^n)$.

If $\dim(A/P)=1$, then as $\dim(A/I)=1$, $P\in \Min(A/I^n)$. Arguing as for case (5), we get $P\in \Ass(I^{n-1}/I^n)$.

If $\dim(A/P)=0$, then $P=\mm$, the graded maximal ideal of $A$. Since $\mm \in \Ass(A/I^n)$, by Lemma \ref{lem_colon_containment}, $\mm \in \Ass(I^{n-1}/I^n)$. 

(8) It is harmless to assume that $I\neq 0$. If $\dim(A/I)\le 1$ then we are done by (7). Assume that $\dim(A/I)\ge 2$, then the hypothesis forces $\dim A=3$ and $\htt I=1$. Thus we can write $I=xL$ where $x$ is a form of degree at least 1, and $L=R$ or $\htt L\ge 2$. The result is clear when $L=R$, so it remains to assume that $L$ is proper of height $\ge 2$. In particular $\dim(A/L)\le 1$, and by (7), for all $n\ge 1$,
\[
\Ass(A/L^n)=\Ass(L^{n-1}/L^n).
\]
Take $\pp \in \Ass(A/I^n)$. Since $A/I^n$ and $I^{n-1}/I^n$ have the same minimal primes, we can assume $\htt \pp \ge 2$. From the exact sequence
\[
0\to A/L^n \xrightarrow{\cdot x^n} A/I^n \to A/(x^n) \to 0
\]
it follows that $\pp \in \Ass(A/L^n)$. Thus $\pp \in \Ass(L^{n-1}/L^n)$. There is an exact sequence
\[
0\to L^{n-1}/L^n \xrightarrow{\cdot x^n}  I^{n-1}/I^n
\]
so $\pp \in \Ass(I^{n-1}/I^n)$, as claimed. This concludes the proof.
\end{proof}

\begin{ex}
Here is an example of a ring $A$ and an ideal $I$ not satisfying any of the conditions (1)--(8) in Theorem \ref{thm_specialcase_ass}.
Let $I=(x^4+y^3z,x^3y,x^2t^2,y^4,y^2z^2) \subseteq A=k[x,y,z,t]$. Then $\depth \gr_I(A)=0$ as $x^2y^3z\in (I^2:I)\setminus I$. So $I$ satisfies neither (1) nor (2).

Note that $\sqrt{I}=(x,y)$, so $\dim(A/I)=2$. Let $\mm=(x,y,z,t)$. Since $x^2y^3zt\in (I:\mm)\setminus I$, $\depth(A/I)=0$, hence $A/I$ is not unmixed. Thus $I$ satisfies neither (5) nor (7). By the proof of Theorem \ref{thm_specialcase_ass}, $I$ satisfies none of the conditions (3), (4), (6). 

Unfortunately, experiments with Macaulay2 \cite{GS} suggest that $I$ satisfies the conclusion of Question \ref{quest_ass}, namely for all $n\ge 1$,
\[
\Ass(A/I^n)=\Ass(I^{n-1}/I^n)=\{(x,y),(x,y,z),(x,y,t),(x,y,z,t)\}.
\]
\end{ex}

\begin{rem}
In view of Lemma \ref{lem_colon_containment} and Question \ref{quest_ass}, one might ask whether if $\chara k=0$, then $\Ass(A/I)=\Ass(\partial(I)/I)$ for any homogeneous ideal $I$ in a polynomial ring $A$?

Unfortunately, this has a negative answer. Let $A=\mathbb{Q}[x,y,z], f=x^5+x^4y+y^4z, L=(x,y)$ and $I=fL$. Then we can check with Macaulay2 \cite{GS} that $\partial(I):f=L$. In particular,
\[
\Ass (\partial(I)/I) =(f) \neq \Ass(A/I)=\{(f),(x,y)\}.
\]
Indeed, if $L=(x,y)\in \Ass(\partial(I)/I)$ then $\Hom_R(R/L,\partial(I)/I)=(\partial(I)\cap (I:L))/I=(\partial(I)\cap (f))/I\neq 0$, so that $\partial(I):f\neq L$, a contradiction.

\end{rem}

\subsection{Partial answer to Question \ref{quest_ass} in dimension four}
We prove that if $\chara k=0$ and $\dim A=4$, the equality $\Ass(A/I^2)=\Ass(I/I^2)$ always holds, in support of a positive answer to Question \ref{quest_ass}. The proof requires the structure theory of perfect Gorenstein ideals of height three and their second powers.
\begin{thm}
\label{thm_dim4_nequals2}
Assume $\chara k=0$. Let $(A,\mm)$ be a four dimensional standard graded polynomial ring over $k$. Then for any proper homogeneous ideal $I$ of $A$, there is an equality $\Ass(A/I^2)=\Ass(I/I^2)$.
\end{thm}
\begin{proof}
It is harmless to assume $I$ is a proper ideal. If $\htt I\ge 3$ then $\dim(A/I)\le 1$, and we are done by Theorem \ref{thm_specialcase_ass}(7).

If $\htt I=1$, then $I=fL$, where $f\in A$ is a form of positive degree and $\htt L\ge 2$. The exact sequence
\[
0 \to \frac{A}{L} \xrightarrow{\cdot f} \frac{A}{I} \to \frac{A}{(f)} \to 0,
\]
yields $\Ass(A/I)=\Ass(A/L) \bigcup \Ass(A/(f))$, as $\Min(I) \supseteq \Ass(A/(f))$. An analogous formula holds for $\Ass(A/I^2)$, as $I^2=f^2L^2$. If we can show that $\Ass(A/L^2) \subseteq \Ass(L/L^2)$, then from the injection $L/L^2 \xrightarrow{\cdot f^2} I/I^2$
we have 
\begin{align*}
\Ass(A/I^2)    &=\Ass(A/L^2) \bigcup \Ass(A/(f)) \\
                        & =\Ass(L/L^2)\bigcup \Ass(A/(f)) \subseteq \Ass(I/I^2).  
\end{align*}
Hence it suffices to consider the case $\htt I=2$. Assume that there exists $\pp \in \Ass(A/I^2) \setminus \Ass(I/I^2)$. The exact sequence
\[
0\to I/I^2 \to A/I^2 \to A/I \to 0
\]
implies $\pp \in \Ass(A/I)$.

By Lemma \ref{lem_colon_containment}, $\pp \neq \mm$. Since $\Min(I) = \Min(I/I^2)$, $\pp \notin \Min(I)$, we get $\htt \pp=3$. Localizing yields $\pp A_\pp \in \Ass(A_\pp/I_\pp^2) \setminus \Ass(I_\pp/I_\pp^2)$. Then there exists $a\in (I_\pp^2:\pp A_\pp) \setminus I_\pp^2$. On the other hand, since $A_\pp$ is a regular local ring of dimension 3 containing one half, Lemma \ref{lem_colon_dim3} below implies $I_\pp^2:\pp A_\pp \subseteq I_\pp$, so $a \in I_\pp \setminus I_\pp^2$. Hence $\pp A_\pp \in \Ass(I_\pp/I_\pp^2)$. This contradiction finishes the proof.
\end{proof}
To finish the proof of Theorem \ref{thm_dim4_nequals2}, we have to show the following
\begin{lem}
\label{lem_colon_dim3}
Let $(R,\mm)$ be a three dimensional regular local ring such that $1/2 \in R$. Then for any ideal $I$ of $R$, there is a containment $I^2:\mm \subseteq I$.
\end{lem}
We will deduce it from the following result.
\begin{prop}
\label{prop_Gorenstein_height3_Torvanishing}
Let $(R,\mm)$ be a regular local ring such that $1/2\in R$. Let $J$ be a perfect Gorenstein ideal of height $3$. Then for all $i\ge 0$, the maps
\[
\Tor^R_i(J^2,k) \to \Tor^R_i(J,k)
\]
is zero. In particular, there is a containment $J^2:\mm \subseteq J$.
\end{prop}
\begin{proof}
We note that the second assertion follows from the first. Indeed, the hypotheses implies that $\dim(R)=d\ge 3$. Using the Koszul complex of $R$, we see that
\[
\Tor^R_{d-1}(J,k) \cong \Tor^R_d(R/J,k) \cong  \frac{J:\mm}{J}.
\]
Since the map $\Tor^R_i(J^2,k) \to \Tor^R_i(J,k)$ is zero for $i=d-1$, the conclusion is $J^2:\mm \subseteq J$. Hence it remains to prove the first assertion. We do this by exploiting the structure of the minimal free resolution of $J$ and $J^2$, and constructing a map between these complexes.

Since $J$ is Gorenstein of height three, it has a minimal free resolution
\[
P: 0\to R \xrightarrow{\delta} F^* \xrightarrow{\rho} F \to 0.
\]
Here $F=Re_1 \oplus \cdots \oplus Re_g$ is a free $R$-module of rank $g$ -- an odd integer. The map $\tau: F\to J$ maps $e_i$ to $f_i$, where $J=(f_1,\ldots,f_g)$. The free $R$-module $F^*$ has basis $e_1^*,\ldots,e_g^*$. The map $\rho: F^* \to F$ is alternating with matrix $(a_{i,j})_{g\times g}$, namely $a_{i,i}=0$ for $1\le i\le g$ and $a_{i,j}=-a_{j,i}$ for $1\le i<j\le g$, and 
\[
\rho(e_i^*)=\sum_{j=1}^g a_{j,i}e_j \quad \text{for all $i$}.
\]
The map $\delta: R\to F^*$ has the matrix $(f_1 \ldots f_g)^T$, i.e. it is given by $\delta(1)=f_1e_1^*+\cdots+f_g e_g^*$.

It is known that if $J$ is Gorenstein of height three, then $J\otimes_R J \cong J^2$, and by constructions due to Kustin and Ulrich \cite[Definition 5.9, Theorems 6.2 and 6.17]{KU}, $J^2$ has a minimal free resolution $Q$ as below. Note that in the terminology of \cite{KU} and thanks to the discussion after Theorem 6.22 in that work, $J$ satisfies $\text{SPC}_{g-2}$, hence Theorem 6.17, parts (c)(i) and (d)(ii) in \emph{ibid.} are applicable. The resolution $Q$ given in the following is taken from (2.7) and Definition 2.15 in Kustin and Ulrich's paper.
\[
Q:   0 \to  \wedge^2 F^*  \xrightarrow{d_2}  (F\otimes F^*)/\eta  \xrightarrow{d_1}    S_2(F)  \xrightarrow{d_0}  J^2 \to   0.
\]
Here $S_2(F)=\bigoplus_{1\le i\le j\le g} R(e_i\otimes e_j)$ is the second symmetric power of $F$, $\eta=R(e_1\otimes e_1^*+\cdots+e_g\otimes e_g^*) \subseteq F\otimes F^*$, and $\wedge^2 F^*$ is the second exterior power of $F^*$.

The maps $d_0: S_2(F)\to J^2$, $d_1: (F\otimes F^*)/\eta \to S_2(F)$, $d_2: \wedge^2 F^*  \to (F\otimes F^*)/\eta$ are given by:
\begin{align*}
d_0(e_i\otimes e_j) &=f_if_j \quad \text{for $1\le i,j\le g$},\\
d_1(e_i\otimes e_j^*+\eta) &=\sum_{l=1}^g a_{l,j} (e_i\otimes e_l) \quad \text{for $1\le i,j\le g$},\\
d_2(e_i^*\wedge e_j^*) &= \sum_{l=1}^g a_{l,i}(e_l\otimes e_j^*) - \sum_{v=1}^g a_{v,j}(e_v\otimes e_i^*)+\eta \quad \text{for $1\le i<j\le g$}.
\end{align*}

We construct a lifting $\alpha: Q\to P$ of the natural inclusion map $J^2\to J$ such that $\alpha(Q)\subseteq \mm P$. 
\[
\xymatrix{
Q:  & 0 \ar[r] &  \wedge^2 F^* \ar[d]^{\alpha_2} \ar[r]^{d_2} & (F\otimes F^*)/\eta \ar[r]^{d_1} \ar[d]^{\alpha_1}  & S_2(F)  \ar[r]^{d_0} \ar[d]^{\alpha_0} & J^2 \ar[r] \ar@{^{(}->}[d]_{\iota}   & 0 \\
P: &  0 \ar[r] & R  \ar[r]^{\delta} & F^* \ar[r]^{\rho} & F \ar[r]^{\tau} & J \ar[r] & 0.
}
\]
In detail, this lifting is

\begin{itemize}
 \item $\alpha_0(e_i\otimes e_j) = \dfrac{f_ie_j+f_je_i}{2} \quad \text{for $1\le i,j\le g$},$
\item $\alpha_1(e_i\otimes e_j^*+\eta)  = \begin{cases}
\dfrac{f_ie_j^*}{2}, & \text{if $(i,j) \neq (g,g)$},\\
\dfrac{-\sum_{v=1}^{g-1}f_v e_v^*}{2}, &\text{if $(i,j)=(g,g)$},
\end{cases}$
\item $\alpha_2(e_i^*\wedge e_j^*) = \begin{cases}
                                0, &\text{if $1\le i<j\le g-1$},\\
                                \dfrac{-a_{g,i}}{2}, &\text{if $1\le i\le g-1, j=g$}.
                \end{cases}$
\end{itemize}
Note that $\alpha_1$ is well-defined since 
$$
\alpha_1(e_1\otimes e_1^*+\cdots+e_g\otimes e_g^*+\eta) = \dfrac{\sum_{v=1}^{g-1}f_v e_v^*}{2}\dfrac{-\sum_{v=1}^{g-1}f_v e_v^*}{2}=0.
$$

Observe that $\alpha(Q)\subseteq \mm P$ since $f_i,a_{i,j} \in \mm$ for all $i,j$. It remains to check that the map $\alpha: Q\to P$ is a lifting for $J^2\hookrightarrow J$. For this, we have:
\begin{itemize}
 \item $\tau(\alpha_0(e_i\otimes e_j))=\tau\left(\dfrac{f_ie_j+f_je_i}{2}\right)=f_if_j=\iota(d_0(e_i\otimes e_j))$.
\end{itemize}
Next we compute
\begin{align*}
\alpha_0(d_1(e_i\otimes e_j^*+\eta)) &=\alpha_0\left(\sum_{l=1}^g a_{l,j} (e_i\otimes e_l)\right)= \sum_{l=1}^g a_{l,j}\dfrac{f_ie_l+f_le_i}{2} \\
                                     &= \dfrac{f_i\left(\sum_{l=1}^ga_{l,j}e_l\right)}{2} \quad \text{(since $\sum_{l=1}^g a_{l,j}f_l=0$)}.
\end{align*}
\begin{itemize}
\item If $(i,j)\neq (g,g)$ then
\begin{align*}
\rho(\alpha_1(e_i\otimes e_j^*+\eta)) &= \rho(f_ie_j^*/2)= \dfrac{f_i\left(\sum_{l=1}^ga_{l,j}e_l\right)}{2}.
\end{align*}
\item  If $(i,j)=(g,g)$ then 
\begin{align*}
\rho(\alpha_1(e_g\otimes e_g^*+\eta)) &= \rho\left(\dfrac{-\sum_{v=1}^{g-1}f_ve_v^*}{2}\right)= \dfrac{-\sum_{v=1}^{g-1}f_v(\sum_{l=1}^g a_{l,v}e_l)}{2}\\
                                    &= \dfrac{\sum\limits_{l=1}^g(\sum\limits_{v=1}^{g-1} a_{v,l}f_v)e_l}{2} \quad (\text{since $a_{v,l}=-a_{l,v}$})\\
                                    &= \dfrac{-\sum\limits_{l=1}^g(a_{g,l}f_g)e_l}{2} \quad \text{(since $\sum_{v=1}^g a_{v,l}f_v=0$)}\\
                                    &= \dfrac{f_g\left(\sum_{l=1}^ga_{l,g}e_l\right)}{2} \quad (\text{since $a_{g,l}=-a_{l,g}$})
\end{align*}
\item Hence in both cases, $\alpha_0(d_1(e_i\otimes e_j^*+\eta))=\rho(\alpha_1(e_i\otimes e_j^*+\eta))$.
\end{itemize}
Next, for $1\le i< j \le g-1$, we compute
\begin{align*}
\alpha_1(d_2(e_i^*\wedge e_j^*)) &= \alpha_1\left(\sum_{l=1}^g a_{l,i}(e_l\otimes e_j^*) - \sum_{v=1}^g a_{v,j}(e_v\otimes e_i^*)+\eta\right)\\
                                &= \dfrac{\left(\sum_{l=1}^g a_{l,i}f_l\right)e_j^*}{2}-\dfrac{\left(\sum_{v=1}^g a_{v,j}f_v\right)e_i^*}{2}\\
                                & \qquad \text{(since neither $(l,j)$ nor $(v,i)$ is $(g,g)$)}\\
                                &=0 \quad \text{(since $\sum_{v=1}^g a_{v,l}f_v=0$)}\\
                                & =\delta(\alpha_2(e_i^*\wedge e_j^*)).
\end{align*}

Finally, for $1\le i\le g-1, j=g$, we have
\begin{align*}
\alpha_1(d_2(e_i^*\wedge e_g^*)) &= \alpha_1\left(\sum_{l=1}^g a_{l,i}(e_l\otimes e_g^*) - \sum_{v=1}^g a_{v,g}(e_v\otimes e_i^*)+\eta\right)\\
                                &= \dfrac{\left(\sum_{l=1}^{g-1} a_{l,i}f_l\right)e_g^*}{2}-\dfrac{\sum_{v=1}^{g-1}a_{g,i}f_v e_v^*}{2} -\dfrac{\left(\sum_{v=1}^g a_{v,g}f_v\right)e_i^*}{2}\\
                                & \text{\footnotesize{(the formula for $\alpha_1(e_l\otimes e_g^*)$ depends on whether $l=g$ or not)}} \\
                                &= \dfrac{-a_{g,i}f_ge_g^*}{2}-\dfrac{\sum_{v=1}^{g-1}a_{g,i}f_v e_v^*}{2} \quad \text{(since $\sum_{v=1}^g a_{v,l}f_v=0$)} \\
                                &= \dfrac{-a_{g,i}\left(\sum_{v=1}^gf_v e_v^*\right)}{2}
\end{align*}
We also have
\[
\delta(\alpha_2(e_i^*\wedge e_g^*))=\delta(-a_{g,i}/2)=\dfrac{-a_{g,i}\left(\sum_{v=1}^gf_v e_v^*\right)}{2}.
\]
Hence $\alpha: Q\to P$ is a lifting of the inclusion map $J^2\to J$. 

Since $\alpha(Q) \subseteq \mm P$, it follows that $\alpha\otimes (R/\mm)=0$. Hence $\Tor^R_i(J^2,k) \to \Tor^R_i(J,k)$ is the zero map for all $i$. The proof is concluded.
\end{proof}
\begin{proof}[Proof of Lemma \ref{lem_colon_dim3}]
It is harmless to assume that $I\subseteq \mm$. We can write $I$ as a finite intersection $I_1\cap \cdots \cap I_d$ of irreducible ideals. If we can show the lemma for each of the components $I_j$, then
\[
I^2:\mm \subseteq (I_1^2:\mm) \cap \cdots \cap (I_d^2:\mm) \subseteq \bigcap_{j=1}^d I_j=I.
\]
Hence we can assume that $I$ is an irreducible ideal. Being irreducible, $I$ is a primary ideal. If $\sqrt{I} \neq \mm$, then $I^2:\mm\subseteq I:\mm=I$. Therefore we assume that $I$ is an $\mm$-primary irreducible ideal. Let $k=R/\mm$. It is a folklore and simple result that any $\mm$-primary irreducible ideal must satisfy $\dim_k (I:\mm)/I=1$. Note that $R$ is a regular local ring, so being a Cohen-Macaulay module of dimension zero, $R/I$ is perfect.  Hence $I$ is a perfect Gorenstein ideal of height three. It then remains to use the second assertion of Proposition \ref{prop_Gorenstein_height3_Torvanishing}.
\end{proof}

In view of Lemma \ref{lem_colon_dim3}, it seems natural to ask the following
\begin{quest}
\label{quest_colon_dim3}
Let $(R,\mm)$ be a three dimensional regular local ring containing $1/2$. Let $I$ be an ideal of $R$. Is it true that for all $n\ge 2$, $I^n:\mm \subseteq I^{n-1}$?
\end{quest}
For regular local rings of dimension at most two, Ahangari Maleki has proved that Question \ref{quest_colon_dim3} has a positive answer regardless of the characteristic \cite[Proof of Theorem 3.7]{AM}. Nevertheless, if $\dim A$ is not fixed, Question \ref{quest_colon_dim3} has a negative answer in positive characteristic in general. Here is a counterexample in dimension 9(!).
\begin{ex}
\label{ex_non-Torvanishing}
Choose $\chara k=2$, $A=k[x_1,x_2,x_3,\ldots,z_1,z_2,z_3]$ and
\[
M=\begin{pmatrix}
   x_1 & x_2 & x_3 \\
   y_1 & y_2 & y_3 \\
   z_1 & z_2 & z_3
  \end{pmatrix}.
\]
Let $I_2(M)$ be the ideal generated by the 2-minors of $M$, and
$$
I=I_2(M)+\sum_{i=1}^3(x_i,y_i,z_i)^2+(x_1,x_2,x_3)^2+(y_1,y_2,y_3)^2+(z_1,z_2,z_3)^2.
$$
Denote $\mm=A_+$. The Betti table of $A/I$, computed by Macaulay2 \cite{GS}, is

\begin{verbatim}
             0  1   2   3   4   5   6   7  8 9
      total: 1 36 160 315 404 404 315 160 36 1
          0: 1  .   .   .   .   .   .   .  . .
          1: . 36 160 315 288 116   .   .  . .
          2: .  .   .   . 116 288 315 160 36 .
          3: .  .   .   .   .   .   .   .  . 1
\end{verbatim}
Therefore $I$ is an $\mm$-primary, binomial, quadratic, Gorenstein ideal. Also, the relation $x_1y_2z_3+x_2y_3z_1+x_3y_1z_2\in (I^2:\mm)\setminus I$ implies $I^2:\mm \not\subseteq I$. This means that the map $\Tor^A_8(k,I^2)\to \Tor^A_8(k,I)$ is not zero. In particular, this gives a negative answer to \cite[Question 3.6]{AM} in positive characteristic.
\end{ex}

\section{Powers of sums and associated primes}
\label{sect_asspr_powers}
\subsection{Bounds for associated primes}

The second main result of this paper is the following. Its part (3) generalizes \cite[Lemma 3.4]{HM}, which deals only with squarefree monomial ideals.
\begin{thm}
\label{thm_Asscontainments}
Let $A, B$ be commutative Noetherian algebras over $k$ such that $R=A\otimes_k B$ is Noetherian. Let $I,J$ be proper ideals of $A,B$, respectively. 
\begin{enumerate}[\quad \rm(1)]
\item For all $n\ge 1$, we have inclusions
\begin{align*}
\bigcup_{i=1}^n \mathop{\bigcup_{\pp \in \Ass_A(I^{i-1}/I^i)}}_{\qq \in \Ass_B(J^{n-i}/J^{n-i+1})} \Min_R(R/\pp+\qq) &  \subseteq   \Ass_R \frac{R}{(I+J)^n},\\
\Ass_R \frac{R}{(I+J)^n} &  \subseteq   \bigcup_{i=1}^n \mathop{\bigcup_{\pp \in \Ass_A(A/I^i)}}_{\qq \in \Ass_B(J^{n-i}/J^{n-i+1})} \Min_R(R/\pp+\qq).
\end{align*}
\item If moreover $\Ass(A/I^n)=\Ass(I^{n-1}/I^n)$ for all $n\ge 1$, then both inclusions in \textup{(1)} are equalities. 

\item In particular, if $A$ and $B$ are polynomial rings and $I$ and $J$ are monomial ideals, then for all $n\ge 1$, we have an equality
\[
\Ass_R \frac{R}{(I+J)^n} = \bigcup_{i=1}^n \mathop{\bigcup_{\pp \in \Ass_A(A/I^i)}}_{\qq \in \Ass_B(B/J^{n-i+1})} \{\pp+\qq\}.
\]
\end{enumerate}
\end{thm}
\begin{proof}
(1) Denote $Q=I+J$. By Proposition \ref{prop_decomposition}, we have
\[
Q^{n-1}/Q^n = \bigoplus_{i=1}^n (I^{i-1}/I^i \otimes_k J^{n-i}/J^{n-i+1}).
\]
Hence
\begin{equation}
\label{eq_inclusion_Ass1}
\bigcup_{i=1}^n \Ass_R (I^{i-1}/I^i \otimes_k J^{n-i}/J^{n-i+1}) = \Ass_R (Q^{n-1}/Q^n) \subseteq \Ass_R(R/Q^n).
\end{equation}
For each $1\le i\le n$, we have $J^{n-i}Q^i\subseteq J^{n-i}(I^i+J)=J^{n-i}I^i+J^{n-i+1}$. We claim that $(J^{n-i}I^i+J^{n-i+1})/J^{n-i}Q^i \cong J^{n-i+1}/J^{n-i+1}Q^{i-1}$, so that there is an exact sequence
\begin{equation}
\label{eq_exactseq}
0\longrightarrow \frac{J^{n-i+1}}{J^{n-i+1}Q^{i-1}} \longrightarrow \frac{J^{n-i}}{J^{n-i}Q^i} \longrightarrow \frac{J^{n-i}}{J^{n-i+1}+J^{n-i}I^i} \cong \frac{A}{I^{i}}\otimes_k \frac{J^{n-i}}{J^{n-i+1}} \longrightarrow 0.
\end{equation}
For the claim, we have
\begin{align*}
(J^{n-i}I^i+J^{n-i+1})/J^{n-i}Q^i &= \frac{J^{n-i}I^i+J^{n-i+1}}{J^{n-i}(I^i+JQ^{i-1})}= \frac{J^{n-i}I^i+J^{n-i+1}}{J^{n-i}I^i+J^{n-i+1}Q^{i-1}}\\
&= \frac{(J^{n-i}I^i+J^{n-i+1})/J^{n-i}I^i}{(J^{n-i}I^i+J^{n-i+1}Q^{i-1})/J^{n-i}I^i}\\
&\cong \frac{J^{n-i+1}/J^{n-i+1}I^i}{J^{n-i+1}Q^{i-1}/J^{n-i+1}I^i} \cong \frac{J^{n-i+1}}{J^{n-i+1}Q^{i-1}}.
\end{align*}
In the display, the first isomorphism follows from the fact that
\[
J^{n-i+1} \cap J^{n-i}I^i=J^{n-i+1}I^i=J^{n-i}I^i\cap J^{n-i+1}Q^{i-1},
\]
which holds since by Proposition \ref{prop_decomposition},
\[
J^{n-i+1}I^i \subseteq J^{n-i}I^i\cap J^{n-i+1}Q^{i-1} \subseteq J^{n-i}I^i \cap J^{n-i+1}  \subseteq I^i\cap J^{n-i+1}=J^{n-i+1}I^i.
\]

Now for $i=n$, the exact sequence \eqref{eq_exactseq} yields
\[
\Ass_R (R/Q^n) \subseteq \Ass_R(J/JQ^{n-1}) \cup \Ass_R (A/I^n\otimes_k B/J).
\]
Similarly for the cases $2\le i\le n-1$ and $i=1$. Putting everything together, 
\begin{equation}
\label{eq_inclusion_Ass2}
\Ass_R (R/Q^n) \subseteq \bigcup_{i=1}^n \Ass_R (A/I^i \otimes_k J^{n-i}/J^{n-i+1}).
\end{equation}
Combining \eqref{eq_inclusion_Ass1}, \eqref{eq_inclusion_Ass2} and Theorem \ref{thm_asso}, we finish the proof of (1).

(2) If $\Ass_A(A/I^n)=\Ass_A(I^{n-1}/I^n)$ for all $n\ge 1$, then clearly the upper bound and lower bound for $\Ass(R/(I+J)^n)$ in part (1) coincide. The conclusion follows.

(3) In this situation, every associated prime of $A/I^i$ is generated by variables. In particular, $\pp+\qq$ is a prime ideal of $R$ for any $\pp \in \Ass(A/I^i), \qq\in \Ass_B(B/J^j)$ and $i, j\ge 1$. The conclusion follows from part (2).
\end{proof}
\begin{rem}
If Question \ref{quest_ass} has a positive answer, then we can strengthen the conclusion of Theorem \ref{thm_Asscontainments}: Let $A, B$ be standard graded polynomial rings over $k$. Let $I,J$ be proper homogeneous ideals of $A,B$, respectively. Then for all $n\ge 1$, there is an equality
\[
\Ass_R \frac{R}{(I+J)^n} = \bigcup_{i=1}^n \mathop{\bigcup_{\pp \in \Ass_A(A/I^i)}}_{\qq \in \Ass_B(B/J^{n-i+1})} \Min(R/(\pp+\qq)).
\]
\end{rem}

\begin{ex}
In general, for singular base rings, each of the inclusions of Theorem \ref{thm_Asscontainments} can be strict. First, take $A=k[a,b,c]/(a^2,ab,ac), I=(b)$, $B=k, J=(0)$. 
Then $R=A, Q=I=(b)$ and $I^2=(b^2)$. Let $\mm=(a,b,c)$. One can check that $a \in (I^2:\mm)\setminus I^2$ and $I/I^2\cong A/(a,b)\cong k[c]$, whence $\depth(A/I^2)=0 <\depth(I/I^2)$. In particular, $\mm\in \Ass_A(A/I^2) \setminus \Ass_A(I/I^2)$. Thus the lower bound for $\Ass_R(R/Q^2)$ is strict in this case.

Second, take $A,I$ as above and $B=k[x,y,z]$, $J=(x^4,x^3y,xy^3,y^4,x^2y^2z).$ In this case $Q=(b,x^4,x^3y,xy^3,y^4,x^2y^2z) \subseteq k[a,b,c,x,y,z]/(a^2,ab,ac)$. Then $c+z$ is $(R/Q^2)$-regular, so $\depth R/Q^2 >0=\depth A/I^2+\depth B/J$. Hence $(a,b,c,x,y,z)$ does not lie in $\Ass_R(R/Q^2)$, but it belongs to the upper bound for $\Ass_R(R/Q^2)$ in Theorem \ref{thm_Asscontainments}(1).
\end{ex}

\subsection{Asymptotic primes}
 Recall that if $I\neq A$, $\grade(I,A)$ denotes the maximal length of an $A$-regular sequence consisting of elements in $I$; and if $I=A$, by convention, $\grade(I,A)=\infty$ (see \cite[Section 1.2]{BH} for more details). Let $\astab^*(I)$ denote the minimal integer $m\ge 1$ such that both $\Ass_A(A/I^i)$ and $\Ass_A(I^{i-1}/I^i)$ are constant sets for all $i\ge m$. By a result due to McAdam and Eakin \cite[Corollary 13]{McE}, for all $i\ge \astab^*(I)$, $\Ass_A(A/I^i)\setminus \Ass_A(I^{i-1}/I^i)$ consists only of prime divisors of $(0)$. Hence if $\grade(I,A)\ge 1$, i.e. $I$ contains a non-zerodivisor, then $\Ass_A(A/I^i)= \Ass_A(I^{i-1}/I^i)$ for all $i\ge \astab^*(I)$. Denote $\Ass_A^*(I)=\bigcup_{i\ge 1} \Ass_A(A/I^i)=\bigcup_{i=1}^{\astab^*(I)}\Ass_A(A/I^i)$ and 
$$
\Ass_A^{\infty}(I)=\Ass_A(A/I^i) \quad \text{for any $i\ge \astab^*(I)$.}
$$ 
The following folklore lemma will be useful.
\begin{lem}
\label{lem_unionAssequal}
For any $n\ge 1$, we have
\[
\bigcup_{i=1}^n \Ass_A(A/I^i)=\bigcup_{i=1}^n \Ass_A(I^{i-1}/I^i).
\]
In particular, if $\grade(I,A)\ge 1$ then 
$$\Ass_A^*(I)=\bigcup_{i=1}^{\astab^*(I)} \Ass_A(I^{i-1}/I^i)=\bigcup_{i\ge 1} \Ass_A(I^{i-1}/I^i).$$
\end{lem}
\begin{proof}
For the first assertion: Clearly the left-hand side contains the right-hand one. Conversely, we deduce from the inclusion $\Ass_A(A/I^i) \subseteq \Ass_A(I^{i-1}/I^i) \cup \Ass_A(A/I^{i-1})$ for $2\le i\le n$ that the other containment is valid as well.

The remaining assertion is clear.
\end{proof}

Now we describe the asymptotic associated primes of $(I+J)^n$ for $n\gg 0$ and provide an upper bound for $\astab^*(I+J)$ under certain conditions on $I$ and $J$.
\begin{thm}
\label{thm_asymptoticAss}
Assume that $\grade(I,A)\ge 1$ and $\grade(J,B)\ge 1$, e.g. $A$ and $B$ are domains and $I, J$ are proper ideals. Then for all $n\ge \astab^*(I)+\astab^*(J)-1$, we have
\begin{align*}
\Ass_R \frac{R}{(I+J)^n}&=\Ass_R \frac{(I+J)^{n-1}}{(I+J)^n} \\
&=  \mathop{\bigcup_{\pp \in \Ass^*_A(I)}}_{\qq \in \Ass^{\infty}_B(J)} \Min_R(R/\pp+\qq) \bigcup \mathop{\bigcup_{\pp \in \Ass^{\infty}_A(I)}}_{\qq \in \Ass^*_B(J)} \Min_R(R/\pp+\qq). 
\end{align*}
In particular, $\astab^*(I+J)\le \astab^*(I)+\astab^*(J)-1$ and
\[
\Ass^{\infty}_R(I+J)= \mathop{\bigcup_{\pp \in \Ass^*_A(I)}}_{\qq \in \Ass^{\infty}_B(J)} \Min_R(R/\pp+\qq) \bigcup \mathop{\bigcup_{\pp \in \Ass^{\infty}_A(I)}}_{\qq \in \Ass^*_B(J)} \Min_R(R/\pp+\qq). 
\]
\end{thm}
\begin{proof}
Denote $Q=I+J$. It suffices to prove that for $n\ge \astab^*(I)+\astab^*(J)-1$, both the lower bound (which is nothing but $\Ass_R(Q^{n-1}/Q^n)$) and upper bound for $\Ass_R(R/Q^n)$ in Theorem \ref{thm_Asscontainments} are equal to 
\[
\mathop{\bigcup_{\pp \in \Ass^*_A(I)}}_{\qq \in \Ass^{\infty}_B(J)} \Min_R(R/\pp+\qq) \bigcup \mathop{\bigcup_{\pp \in \Ass^{\infty}_A(I)}}_{\qq \in \Ass^*_B(J)} \Min_R(R/\pp+\qq).
\]
First, for the lower bound, we need to show that for $n\ge \astab^*(I)+\astab^*(J)-1$,
\begin{align}
&\bigcup_{i=1}^n \mathop{\bigcup_{\pp \in \Ass_A(I^{i-1}/I^i)}}_{\qq \in \Ass_B(J^{n-i}/J^{n-i+1})} \Min_R(R/\pp+\qq) \label{eq_asymptoticlowerboundofAss}\\
&= \mathop{\bigcup_{\pp \in \Ass^*_A(I)}}_{\qq \in \Ass^{\infty}_B(J)} \Min_R(R/\pp+\qq) \bigcup \mathop{\bigcup_{\pp \in \Ass^{\infty}_A(I)}}_{\qq \in \Ass^*_B(J)} \Min_R(R/\pp+\qq). \nonumber
\end{align}
If $i\le \astab^*(I)$, $n-i+1\ge \astab^*(J)$, hence $\Ass_B(J^{n-i}/J^{n-i+1})=\Ass^{\infty}_B(J)$. In particular,
\begin{align*}
&\bigcup_{i=1}^{\astab^*(I)} \mathop{\bigcup_{\pp \in \Ass_A(I^{i-1}/I^i)}}_{\qq \in \Ass_B(J^{n-i}/J^{n-i+1})} \Min_R(R/\pp+\qq) \\
&=\bigcup_{i=1}^{\astab^*(I)} \mathop{\bigcup_{\pp \in \Ass_A(I^{i-1}/I^i)}}_{\qq \in \Ass^{\infty}_B(J)} \Min_R(R/\pp+\qq)=\mathop{\bigcup_{\pp \in \Ass^*_A(I)}}_{\qq \in \Ass^{\infty}_B(J)} \Min_R(R/\pp+\qq),
\end{align*}
where the second equality follows from Lemma \ref{lem_unionAssequal}.

If $i\ge \astab^*(I)$ then $\Ass_A(A/I^i)=\Ass^{\infty}_A(I)$, $1\le n+1-i \le n+1-\astab^*(I)$. Hence
\begin{align*}
&\bigcup_{i=\astab^*(I)}^{n} \mathop{\bigcup_{\pp \in \Ass_A(I^{i-1}/I^i)}}_{\qq \in \Ass_B(J^{n-i}/J^{n-i+1})} \Min_R(R/\pp+\qq) \\
&=\bigcup_{i=1}^{n+1-\astab^*(I)} \mathop{\bigcup_{\pp \in \Ass^{\infty}_A(I)}}_{\qq \in \Ass_B(J^{i-1}/J^i)} \Min_R(R/\pp+\qq)=\mathop{\bigcup_{\pp \in \Ass^{\infty}_A(I)}}_{\qq \in \Ass^*_B(J)} \Min_R(R/\pp+\qq).
\end{align*}
The second equality follows from the inequality $n+1-\astab^*(I)\ge \astab^*(J)$ and Lemma \ref{lem_unionAssequal}. Putting everything together, we get \eqref{eq_asymptoticlowerboundofAss}. The argument for the equality of the upper bound is entirely similar. The proof is concluded.
\end{proof}

\section{The persistence property of sums}
\label{sect_persistence}

Recall that an ideal $I$ in a Noetherian ring $A$ has the \emph{persistence property} if $\Ass(A/I^n)\subseteq \Ass(A/I^{n+1})$ for all $n\ge 1$.
There exist ideals which fail the persistence property: A well-known example is $I=(a^4,a^3b,ab^3,b^4,a^2b^2c)\subseteq k[a,b,c]$, for which $I^n=(a,b)^{4n}$ and $(a,b,c) \in \Ass(A/I) \setminus \Ass(A/I^n)$ for all $n\ge 2$. (For the equality $I^n=(a,b)^{4n}$ for all $n\ge 2$, note that 
$$
U=(a^4,a^3b,ab^3,b^4) \subseteq I \subseteq (a,b)^4.
$$ 
Hence $U^n\subseteq I^n\subseteq (a,b)^{4n}$ for all $n$, and it remains to check that $U^n=(a,b)^{4n}$ for all $n\ge 2$. By direct inspection, this holds for $n\in \{2,3\}$. For $n\ge 4$, since $U^n=U^2U^{n-2}$, we are done by induction.) However, in contrast to the case of monomial ideals, it is still challenging to find a homogeneous prime ideal without the persistence property (if it exists).

Swanson and R. Walker raised the  question \cite[Question 1.6]{SW} whether given two ideals $I$ and $J$ living in different polynomial rings, if both of them have the persistence property, so does $I+J$. The third main result answers in the positive \cite[Question 1.6]{SW} in many new cases. In fact, its case (ii) subsumes \cite[Corollary 1.7]{SW}. 

\begin{cor}
\label{cor_persistence}
Let $A$ and $B$ be standard graded polynomial rings over $k$, $I$ and $J$ are proper homogeneous ideals of $A$ and $B$, respectively. Assume that $I$ has the persistence property, and $\Ass(A/I^n)=\Ass(I^{n-1}/I^n)$ for all $n\ge 1$. Then $I+J$ has the persistence property. In particular, this is the case if any of the conditions hold:
\begin{enumerate}[\quad \rm(i)]
 \item $I$ is a monomial ideal satisfying the persistence property;
 \item $I^{n+1}:I=I^n$ for all $n\ge 1$. 
 \item $I^n$ is unmixed for all $n\ge 1$.
 \item $\chara k=0$, $\dim(A/I)\le 1$ and $I$ has the persistence property.
\end{enumerate}
\end{cor}
\begin{proof}
The hypothesis $\Ass(A/I^n)=\Ass(I^{n-1}/I^n)$ for all $n\ge 1$ and Theorem \ref{thm_Asscontainments}(2), yields for all such $n$ an equality
\begin{equation}
\label{eq_Ass_formula}
\Ass_R \frac{R}{(I+J)^n}  =   \bigcup_{i=1}^n \mathop{\bigcup_{\pp \in \Ass_A(A/I^i)}}_{\qq \in \Ass_B(J^{n-i}/J^{n-i+1})} \Min_R(R/\pp+\qq).
\end{equation}
Take $P \in \Ass_R \dfrac{R}{(I+J)^n}$, then for some $1\le i\le n,\pp \in \Ass_A(A/I^i)$ and $\qq \in \Ass_B(J^{n-i}/J^{n-i+1})$, we get $P\in  \Min_R(R/\pp+\qq)$.

Since $I$ has the persistence property, it follows that $\Ass(A/I^i)\subseteq \Ass(A/I^{i+1})$, so $\pp\in \Ass(A/I^{i+1})$. Hence thanks to \eqref{eq_Ass_formula},
\[
P \in \mathop{\bigcup_{\pp_1 \in \Ass_A(A/I^{i+1})}}_{\qq_1 \in \Ass_B(J^{n-i}/J^{n-i+1})} \Min_R(R/\pp_1+\qq_1) \subseteq \Ass_R \frac{R}{(I+J)^{n+1}}.
\]
Therefore $I+J$ has the persistence property.

The second assertion is a consequence of the first, Theorem \ref{thm_specialcase_ass} and Lemma \ref{lem_sufficient_forpersistence}. 
\end{proof}

%%%%%%%%
\section*{Acknowledgments}
The first author  (HDN) and the second author (QHT) are supported by Vietnam National Foundation for Science and Technology Development (NAFOSTED) under grant numbers 101.04-2019.313 and 1/2020/STS02, respectively. Nguyen is also grateful to the support of Project CT0000.03/19-21 of the Vietnam Academy of Science and Technology. Tran also acknowledges the partial support of the Core Research Program of Hue University, Grant No. NCM.DHH.2020.15. Finally, part of this work was done while the second author was visiting the Vietnam Institute for Advance Study in Mathematics (VIASM). He would like to thank the VIASM for its hospitality and financial support. 

The authors are grateful to the anonymous referee for his/her careful reading of the manuscript and many useful suggestions.

%%%%%%%%%%%%%%%%%%%%%%%%%%%%%%%%%%%%%%%%%%%%%%%%%%%%%%%%%%%%%%%%%%%%%%%%%%%%%%

\end{document}